\documentclass{amsart}
\usepackage{enumerate}
\usepackage{mathdots}
\newtheorem*{acknowledgements}{Acknowledgements}

\def\Ad{\mathrm{Ad}}

\def\CC{\mathbb{C}}
\def\RR{\mathbb{R}}
\def\Ha{\mathbb{H}}
\def\dist{\mbox{-}\mathrm{dist}}
\def\cD{\mathcal{D}}

\def\S{\mathrm{S}}
\def\SD{\mathcal{SD}}
\def\Sc{\mathcal{S}}
\def\nn{\overline{n}}

\def\inv{\mathrm{inv}}

\def\bG{\boldsymbol{\G}}
\def\DD{\Delta}
\def\SF{\mathcal{SAF}}
\def\bH{\boldsymbol{\H}}
\def\tK{\widetilde{\K}}

\usepackage{amsmath,amssymb,mathrsfs}
\usepackage{xcolor}

\newcounter{makeconstant}
\newenvironment{makeconstant}%
{\refstepcounter{makeconstant}}%
{}

\theoremstyle{plain}
\newtheorem{theorem}[equation]{Theorem}
\newtheorem{lemma}[equation]{Lemma}
\newtheorem{corollary}[equation]{Corollary}
\newtheorem{proposition}[equation]{Proposition}

\theoremstyle{definition}
\newtheorem{example}[equation]{Example}
\newtheorem{definition}[equation]{Definition}

\newtheorem{remark}[equation]{Remark}

\newtheorem*{remark*}{Remark}

\numberwithin{equation}{section} 

\def\A{{\rm A}}

\def\C{{\rm C}}
\def\D{{\rm D}}
\def\E{{\rm E}}
\def\F{{\rm F}}
\def\G{{\rm G}}
\def\H{{\rm H}}
\def\I{{\rm I}}
\def\J{{\rm J}}
\def\K{{\rm K}}
\def\L{{\rm L}}
\def\M{{\rm M}}
\def\N{{\rm N}}

\def\P{{\rm P}}
\def\Q{{\rm Q}}
\def\R{{\rm R}}

\def\U{{\rm U}}

\def\d{\delta}
\def\e{\epsilon}

\def\({\left(}
\def\){\right)}
\def\>{\geqslant}
\def\<{\leqslant}

\def\Hom{\operatorname{Hom}}

\def\GL{\operatorname{GL}}
\def\Gal{\operatorname{Gal}}

\def\Res{\operatorname{Res}}

\def\Irr{\operatorname{Irr}}

\def\dim{\operatorname{dim}}

\def\diag{\operatorname{diag}}

\def\tG{{\widetilde{\G}}}
\def\tH{{\widetilde{\H}}}

\def\presuper#1#2%
  {\mathop{}%
   \mathopen{\vphantom{#2}}^{#1}%
   \kern-\scriptspace%
   #2}

\title{Multiplicity one for pairs of Prasad--Takloo-Bighash type}
\author{Paul Broussous}
\address{Universit\'e de Poitiers, Laboratoire de Math\'ematiques et Applications,
T\'el\'eport 2 - BP 30179, Boulevard Marie et Pierre Curie, 86962, Futuroscope Chasseneuil Cedex. France.}
\email{Paul.Broussous@math.univ-poitiers.fr}
\author{Nadir Matringe}
\address{Universit\'e de Poitiers, Laboratoire de Math\'ematiques et Applications,
T\'el\'eport 2 - BP 30179, Boulevard Marie et Pierre Curie, 86962, Futuroscope Chasseneuil Cedex. France.}
\email{nadir.matringe@math.univ-poitiers.fr}

\begin{document}
\maketitle

\begin{abstract}
Let $\E/\F$ be a quadratic extension of non-archimedean local fields of characteristic different from $2$. 
Let $\A$ be an $\F$-central simple algebra of even dimension so that it contains $\E$ as a subfield, 
set $\G=\A^\times$ and $\H$ for the centralizer of $\E^\times$ in $\G$. Using a Galois descent argument, we prove that all double cosets $\H g \H\subset \G$ are stable under the anti-involution 
$g\mapsto g^{-1}$, reducing to Guo's result for $\F$-split $\G$ (\cite{Guo97}) which we extend to fields of positive characteristic different from $2$. 
We then show, combining global and local results, that $\H$-distinguished irreducible representations of $\G$ are self-dual and this implies that $(\G,\H)$ is a Gelfand pair: \[\dim_{\mathbb{C}}(\Hom_{\H}(\pi,\mathbb{C}))\leq 1\] 
for all smooth irreducible representations $\pi$ of $\G$. Finally we explain how to obtain the the multiplicity one statement in the archimedean case using the criteria of Aizenbud and Gourevitch (\cite{AG09}), and we then show self-duality of irreducible distinguished representations in the archimedean case too.
\end{abstract}

\section{Introduction}\label{section intro}

Let $\E/\F$ be a quadratic extension of non-archimedean local fields of characteristic not $2$. Let us set  
$\G=\A^\times$ for $\A$ a central simple $\F$-algebra of even dimension $n=2m$ and $\H$ the centralizer in $\G$ of 
$\E$ embedded in $\A$ as an $\F$-sub-algebra (all such embeddings are $\A^\times$-conjugate). If $\pi$ is a smooth irreducible representation
of $\G$ (especially a discrete series representation and more generally a representation of $\G$ with generic transfer to its $\F$-split form), there are fine conjectures 
of Prasad and Takloo-Bighash (\cite{PTB11}, Conjecture 1) which predict in terms of its Langlands parameter when $\pi$ 
should be $\H$-distinguished: when is the space $\Hom_{\H}(\pi,\chi)$ not reduced to zero when $\chi$ is a character of $\H$?
These conjectures are inspired by earlier works of Tunnell and Saito (\cite{T83} and \cite{S93}) on $\GL(2,\F)$, and there 
has been recent progress made towards a positive answer to them in several cases (\cite{FMW18}, \cite{C18} and \cite{X18}). We will say that the pair $(\G,\H)$ described above is of Prasad--Takloo-Bighash type, or PTB type in short.

One basic question which has still not been addressed in general for such pairs is multiplicity one: does one have \[\dim(\Hom_{\H}(\pi,\chi))\leq 1\] for all irreducible representations $\pi$ of $\G$? When $\chi=\mathbf{1}$, $\F$ is of characteristic zero and $\H\simeq \GL(m,\E)$, the answer is yes thanks to Guo's work \cite{Guo97}. In this paper, after extending Guo's result to non-archimedean local fields of characteristic different from $2$, we deduce that $\dim(\Hom_{\H}(\pi,\mathbf{1}))\leq 1$ for any pair of PTB type $(\G,\H)$ whenever $\pi$ is an irreducible representation of $\G$.

The main result of \cite{Guo97} is that the double cosets $\H g\H\subset \G$ are fixed by the anti-involution $i:g\mapsto g^{-1}$ of $\G$ from which he deduces by the Gelfand-Kazhdan method (\cite{GK75}, or more accuratly \cite{BZ76}) that 
\begin{equation}\label{equation GP2} \dim_{\mathbb{C}}(\Hom_{\H}(\pi,\mathbb{C}))\dim_{\mathbb{C}}(\Hom_{\H}(\pi^\vee,\mathbb{C}))\leq 1\end{equation} for 
any smooth irreducible representation $\pi$ of $\G$ with contragredient $\pi^\vee$. This is enough for him to get multiplicity one in the cases he considers because the group $\H$ is naturally stabilized by an anti-involution $\tau$ of $\G$ such that 
$\pi^\vee\simeq \pi\circ \tau$.  In fact the existence of $\tau$ and the Gelfand-Kazhdan arguments even imply that when $\pi$ is $\H$-distinguished, it is self-dual: $\pi\simeq \pi^\vee$. Note that Guo's work is inspired by the work \cite{JR96} of Jacquet and Rallis, where the authors prove multpilicity one and self-duality of irreducible distinguished representations 
for the pair $(\GL_n(\F),\GL_m(\F)\times \GL_m(\F))$.

Here we deduce from Guo's paper by a Galois descent argument that the double cosets $\H g\H\subset \G$ are always fixed by $i$, and inequality (\ref{equation GP2}) 
follows (Sections \ref{section double cosets}). But then we do not have the anti-involution $\tau$ at our disposal anymore. In Section \ref{section self-duality and multiplicity one}, we first 
deduce that an $\H$-distinguished representation $\pi$ is self-dual when it is cuspidal by globalizing it as a local component of a distinguished cuspidal automorphic representation (\cite{PSP08} and \cite{GL18}), 
strong multiplicity one (\cite{BaduGJL08} and \cite{BR17}) and the results of \cite{Guo97} and 
\cite{JR96}. We then extend this self-duality result to the case of distinguished standard modules, hence to that of irreducible representations, by standard Mackey theory arguments. 
The multiplicity one result follows (Theorem \ref{theorem self duality of dist irreps} and Corollary \ref{corollary PTB is G}).

Finally in Section \ref{section archimedean case}, we explain how the multiplicity one statement in the archimedean case follows 
from the criteria given in \cite{AG09}, and we adapt the argument of Jacquet-Rallis to the archimedean setting to show self-duality of irreducible distinguished representations.

We conclude this introduction by noticing that together with the results of \cite{JR96} and \cite{AG09} in the archimedean case, the results above imply that over global fields of characteristic different from $2$, the global periods of cuspidal automorphic representations for pairs of PTB type are products of local periods.

{\acknowledgements{We thank A. Bouaziz for useful conversations, and D. Gourevitch for useful explanations concerning \cite{AG09}. We thank the referees for their useful comments, leading to clarifications of some parts of the paper.}}

\section{Notation and preliminaries}\label{section notation}

If $\K$ is a group, $X\subset \K$ and $\K'$ is a subgroup of $\K$, we denote by $\K'_X$ the centralizer of $X$ in $\K'$. If 
$X=\{k\}$ we shall write $\K'_k$  for $\K'_{\{k\}}$. 

If $\R$ is a ring, $Y\subset \R$ and $\R'$ is a subring of $\R$, we denote by $\R'_Y$ the centralizer of $Y$ in $\R'$. If 
$Y=\{r\}$ we shall write $\R'_r$ for $\R'_{\{r\}}$. 

If $\Gamma$ is a group acting on a set $Z$, we denote by $Z^\Gamma$ the set of elements in $Z$ fixed by $\Gamma$.

If $l$ is field with separable closure $l^s$ and $l'$ is a Galois extension of $l$ inside $l^s$, we denote by 
$\Gal(l'/l)$ its Galois group. If $\bG$ is an algebraic group defined over $l$,  we denote by $\H^1(\Gal(l'/l),\bG)$ the first cohomology set 
of $\Gal(l'/l)$ with values in $\bG$. When $l'=l^s$ we shall set $\H^1(l,\bG):= \H^1(\Gal(l^s/l),\bG)$. We will make essential use of the following
well-known property of certain cohomology sets of this type, proved for example in \cite[Section 1.7, Example 1]{Kneser69}:

\begin{lemma}\label{lemma trivial cohomology}
Let $k$ be a field, $\A$ be a be a finite dimensional $k$-algebra, then for any finite Galois extension $k'$ of $k$, one has 
\[\H^1(\Gal(k'/k), (\A\otimes_k k')^\times)=\{1\}.\]
\end{lemma}

We denote by $\F$ a local field (non-archimedean or archimedean) of characteristic different from $2$ and by $|.|$ its normalized absolute value. We denote by $\nu$ 
the absolute value of the reduced norm on any central simple $\F$-algebra. We denote by $\E$ a quadratic extension 
of $\F$ and by $|.|_\E$ its normalized absolute value (in particular in the archimedean case $\F=\RR$ and $\E=\CC$). 

\subsection{Central simple algebras over local fields}\label{subsection CSA}

We denote by $\A$ a central simple $\F$-algebra of \textit{even  dimension}. It
is of the form $\mathcal{M}_n(\D)$ for $\D$ a central division $\F$-algebra. We denote by $d$ the square root of the dimension of $\D$ over $\F$ and call it the index of
$\A$; the fact that $\dim_\F(\A)$ is even thus translates as $nd$ is even. Note that if $d$ is odd, then 
$\D\otimes_\F \E$ is a central division $\E$-algebra which we denote by $\D^{\E}$. If $d$ is even, then 
$\E$ embeds in $\D$ and $\D_\E$ is also a central division $\E$-algebra. In any case, whether $d$ is odd or even, the field $\E$ embeds as 
an $\F$-sub-algebra of $\A$ (all such embeddings being conjugate by $\A^\times$ thanks to the Skolem-Noether theorem), and 
$\A_{\E} \underset{\E}{\simeq} \mathcal{M}_n(\D_\E)$ if $d$ is even, whereas 
$\A_{\E} \underset{\E}{\simeq} \mathcal{M}_m(\D^\E)$ if $d$ is odd and $n=2m$.

\subsection{Symmetric pairs}\label{subsection symmetric pairs}

\begin{definition}\label{definition symmetric pair}
Let $\bG$ be a reductive group defined over $\F$ with an $\F$-rational involution $\theta$. We set $\bH=\bG^\theta$. 
We call the triple $(\bG,\bH,\theta)$ an \textit{$\F$-symmetric pair} or just a \textit{symmetric pair} if we do not want to take the non-archimedean local field $\F$ into account. We say that $(\G,\H)$ is of \textit{symmetric type}.
\end{definition}

The following are two notorious examples of symmetric pairs.

\begin{example}\label{example sym pairs}
\noindent \begin{enumerate}
\item\label{example diagonal pair} This is the most simple example. Take $\bH$ an $\F$-reductive group, $\bG=\bH\times \bH$ and $\theta$ the switching involution $(x,y)\mapsto (y,x)$ which fixes the 
diagonal embedding $\Delta(\bH)$ of $\bH$ into $\bG$. Then 
$(\bG,\Delta(\bH),\theta)$ is an $\F$-symmetric that we call a \textit{diagonal pair}.
\item\label{example galois pair} Let $\bH$ be an $\F$-reductive group, $\bG=\Res_{\E/\F}(\bH)$ the Weil restriction of $\bH$ with respect to the extension $\E/\F$, and $\theta$ the 
involution induced by the generator of the galois group $\Gal(\E/\F)$. Then 
$(\bG,\bH,\theta)$ is an $\F$-symmetric pair called a \textit{Galois pair}.
\end{enumerate}
\end{example}

Now we describe the pair of main interest to us. For this purpose we write $\E=\F(\d)$ with $\d^2\in \F$.

\begin{definition}\label{definition PTB pair}
There is (up to $\F$-isomorphism) a unique $\F$-symmetric pair $(\bG,\bH,\theta)$ such that $\G=\A^\times$, $\theta=\mathrm{Ad}(\d)$ (i.e. the conjugation $x\mapsto \d.x.\d^{-1}$), hence $\H=\A_\E^\times$. We call such a pair a \textit{PTB pair} (a Prasad--Takloo-Bighash pair). Moreover we will say that 
the pair $(\G,\H)$ is of \textit{PTB type}. By definition the index $d$ of $(\G,\H)$ will be that of $\A$. When $d=1$, we shall say pairs of \textit{Guo type}.
\end{definition}

\begin{remark}
Note that Guo more generaly considers PTB pairs with $d=1$ or $2$, but for us it will be convenient to exclude the second case of our definition of "Guo pairs", as it will not play any particular role here.
\end{remark}

We shall also come across the following closely related pair.

\begin{definition}\label{definition JR pair}
We call the $\F$-symmetric pair $(\GL_{2n},\GL_n\times \GL_n,\theta_{n,n})$ with \[\theta_{n,n}=\Ad\begin{pmatrix}I_n & \\ & -I_n  \end{pmatrix}\] a \textit{JR pair} (a Jacquet--Rallis pair). Moreover we will say that 
the pair $(\GL_{2n}(\F),\GL_n(\F)\times \GL_n(\F))$ is of \textit{JR type}. 
\end{definition}

\section{Pairs of Gelfand-Kazhdan and of Gelfand type}

From now on, and untill Section \ref{section archimedean case} which is concerned with the archimedean setting, we focus on the non archimedean case. Here we recall the main idea from \cite{GK75}, which allows to reduce multiplicity one statements for 
irreducible representations of $p$-adic groups to statements on invariant distributions over such groups. 

Let $\G$ be an l-group (locally compact totally disconnected). For $\pi$ a smooth representation of 
$\G$, we denote by 
$\pi^\vee$ the smooth contragredient of $\pi$. We denote by $\Irr(\G)$ the set isomorphism classes of smooth admissible irreducible representations of $\G$, 
and by $\Irr_{\H\dist}(\G)$ that of isomorphism 
classes of $\H$-distinguished representations of $\G$ inside $\Irr(\G)$, i.e. those $\pi$ which satisfy \[\Hom_{\H}(\pi,\CC)\neq \{0\}.\]

Denoting by $\mathcal{C}_c^\infty(\G)$ the space of smooth functions from 
$\G$ to $\CC$ with compact support, we set \[\mathcal{D}(\G)=\Hom_{\CC}(\mathcal{C}_c^\infty(\G),\CC)\] and call it the space of distributions on $\G$. Note that 
$\mathcal{C}_c^\infty(\G)$ is naturally equipped with actions of $\G$ by left and right translations, and it thus makes 
sense to talk about distributions on $\G$ invariant on the left or on the right, under the action of a subgroup of $\G$. Similarly, 
any bicontinuous (anti)-automorphism of $\G$ gives birth to an automorphism of $\mathcal{C}_c^\infty(\G)$, hence of $\cD(\G)$.

\begin{definition}\label{definition GK pair}
Let $\G$ be an l-group and $\H$ be a closed subgroup of $\G$. We say that the pair $(\G,\H)$ is of \textit{Gelfand-Kazhdan type}, or GK type in short, if 
there exists a continuous anti-involution $i$ of $\G$ such that any $\H$-bi-invariant distribution in $\mathcal{D}(\G)$ is fixed by $i$.
\end{definition}

A form of the following result can be found in \cite{GK75}; we state it as in \cite[Lemma 4.2]{PrasadDist}.

\begin{proposition}\label{proposition GK type implies GP2}
If $(\G,\H)$ is of GK type, then:
\[\dim_{\mathbb{C}}(\Hom_{\H}(\pi,\mathbb{C})).\dim_{\mathbb{C}}(\Hom_{\H}(\pi^\vee,\mathbb{C}))\leq 1\] for any 
$\pi\in \Irr(\G)$.
\end{proposition}

When $\F$ has characteristic zero, general criteria are given in \cite{AG09} to check that a pair $(\G,\H)$ of symmetric type is of GK type, one can check with some work 
that they do apply in the case of pairs of PTB type, and we will in fact use them in Section \ref{section archimedean case}. However, there is one simple case where a shorter proof can be given, which is when the double cosets 
$\H g \H$ are stable under an anti-involution $i$. 

\begin{proposition}\label{proposition double cosets stability implies GK type}
Let $(\G,\H)$ be of symmetric type. Suppose moreover that there exists a continuous anti-involution $i$ of $\G$ such that 
$i(\H g \H)=\H g \H$ for all $g\in \G$, then $(\G,\H)$ is of GK type.
\end{proposition}
\begin{proof}
This is a consequence of \cite[Theorems 6.13 and 6.15]{BZ76} (which provide a strengthened version of the classical result of \cite{GK75})
applied to the natural action of $\H\times \H$ on $\G$ and the homeomorphism $i$ of $\G$.
\end{proof}

In general one is more interested in multiplicity one. To this end we make the following definition.

\begin{definition}\label{definition G pair}
Let $\G$ be an l-group and $\H$ a closed subgroup of $\G$. We say that $(\G,\H)$ is of \textit{Gelfand type} if 
\[\dim_{\mathbb{C}}(\Hom_{\H}(\pi,\mathbb{C}))\leq 1\] for any $\pi\in \Irr(\G)$.
\end{definition}

Note that Proposition \ref{proposition GK type implies GP2} is not enough to conclude that a pair of GK type is of Gelfand type. However it obviously is when 
$\pi$ and $\pi^\vee$ are $\H$-distinguished 
together for all $\pi\in \Irr(\G)$. We will see in Section \ref{section double cosets} that pairs of PTB type are of GK type, and in 
Section \ref{section self-duality and multiplicity one} that if $(\G,\H)$ is of PTB type, then all $\pi\in \Irr_{\H\dist}(\G)$ are self-dual.

\section{Stability of the double cosets under the anti-involution}\label{section double cosets}

The main result of \cite{Guo97} is that if $(\G,\H)$ is of Guo type, the set of double cosets $\H g \H\subset \G$ are stabilized by $i:g\mapsto g^{-1}$. 
In this section we deduce from Guo's result that this property remains true for any PTB pair by a Galois descent argument. First, we explain the modifications needed in Guo's proof to see that it remains valid as soon as the characteristic of $\F$ is not $2$.

\subsection{Guo's result in positive characteristic}

We recall that Guo considers PTB pairs of index $1$ and $2$, and that we only need to consider the first case for our purpose. So when 
$d=1$, we will here extend Guo's result on stability of double cosest to any field of characteristic different from $2$, 
where Guo only considers the characteristic zero case. 

The only places of \cite{Guo97} which apparently require characteristic zero are \cite[Lemma 3.2 (1) and (2) and second statement of Lemma 3.3]{Guo97} which make use of the exponential map.
We will show that when the characteristic of $\F$ is different from $2$, these lemma can be proved without this tool in a simpler manner. 
We first recall Guo's notations. Because we are only interested in the case $d=1$, we realize $\G$ as the group of matrices given in block form by 
\[\G=\{ g(\alpha,\beta)=\begin{pmatrix} \alpha & \beta \\ \overline{\beta} & \overline{\alpha} \end{pmatrix} \in \GL_{2n}(\E)\},\] where 
$x\mapsto \overline{x}$ is the conjugation of $\E/\F$. In particular this means that Guo's element $\gamma$ (\cite[p. 276]{Guo97}) is equal to $1$.  
The group $\H$ then becomes \[\H=\{ g(\alpha,0)\in \G\}.\] We further introduce the matrix \[\epsilon=\d.\begin{pmatrix} \I_n & 0 \\ 0 & -\I_n \end{pmatrix}\in \H.\] In particular 
$\H$ is the set of matrices in $\G$ commuting 
with $\e$. We introduce the symmetric space 
\[\S =\{s\in \G, \e s \e^{-1}=s^{-1}\}.\] The map \[\rho: g\mapsto g\epsilon g^{-1} \epsilon^{-1}\] induces a homeomorphism between 
$\G/\H$ and $\S$, as shown in \cite[p. 282]{Guo97} and $\S$ is stable under $\H$-conjugation because for example 
\[\rho(h_1 g h_2)=h_1\rho(g)h_1^{-1}\] for $h_i\in \H$. In particular two elements of $\G$ are in the same $\H$-double coset if and only if their image under $\rho$ are $\H$-conjugate. 

We recall from \cite[p. 282]{Guo97} again that \begin{equation}\label{equation equations of S} g(\alpha,\beta)\in \S \Leftrightarrow \alpha\beta=\beta \overline{\alpha} \ \& \ \alpha^2=\I_n +\beta \overline{\beta}.\end{equation}  
We denote by $\U$ the set of unipotent elements in $\G$ and put $\U_\S=\U \cap \S$. We now prove \cite[Lemma 3.2 (2)]{Guo97} for fields of characteristic not $2$; \cite[Lemma 3.2 (1)]{Guo97} is only needed in \cite[Lemma 3.3]{Guo97} but we will reprove this lemma as well without using it. So what we want to prove is:

\begin{lemma}{\cite[Lemma 3.2 (2)]{Guo97}}\label{lemma Guo 3.2 (2)}
One has \[\H \U_{\S} \H=\{g\in \G,\ \rho(g)\in \U_{\S}\}\] and all $\H$-double cosets in this set are stable under $i$.
\end{lemma}
\begin{proof}
Take $u\in \U_\S$. Then $\rho(u)=u^2\in \U_\S$ and this implies that $\H \U_{\S} \H\subset \{g\in \G,\ \rho(g)\in \U_{\S}\}$ because $\rho(h_1 g h_2)=h_1\rho(g) h_1^{-1}$ for $h_i\in \H$. For the converse we first observe that because $\F$ has characteristic different from $2$, the map $\mathrm{sq}:x\mapsto x^2$ is a bijection from $\U$ to itself. This implies that $\mathrm{sq}$ restricts as a bijection from $\U_\S$ to itself: indeed if $\e x\e^{-1}=x^{-1}$ clearly $\e x^2\e^{-1}=x^{-2}$ hence $\mathrm{sq}$ stabilizes $\U_\S$, on the other hand suppose that $x\in \U_\S$, then $x=y^2$ for a unique $y\in \U$ but then 
$(\e y\e^{-1})^{-1}$ is also a square root of $x$ in $\U$ hence it is equal to $y$ so $y\in \U_\S$. Now take $g$ such that $\rho(g)=u\in \U_{\S}$, in particular $u=v^2$ for a unique $v\in \U_\S$, but then $\rho(g)=\rho(v)$ so $g\in \H v\H\subset \H \U_{\S} \H$. 
Finally for $u\in \U_\S$, one has $u^{-1}=\e u\e^{-1}\in \H u \H$, from which the second statement follows.
\end{proof}

It remains prove \cite[second statement of Lemma 3.3]{Guo97} for fields of characteristic not $2$. We put 
\[w=\begin{pmatrix} 0 & I_n \\ I_n & 0 \end{pmatrix}\] so that $w$ is of order $2$, $w\e w^{-1}=-\e$ and $w$ normalizes $\H$. We first prove the following result which does not appear in \cite{Guo97}.

\begin{proposition}\label{proposition Guo prop 2 mais plus fort}
For $u\in \U_\S$, then $\overline{u}=wuw$ is $\H$-conjugate to $u$.
\end{proposition}
\begin{proof}
We write $u=g(\alpha,\beta)$. By \cite[Lemma 3.2 (3)]{Guo97}, the proof of which is valid in characteristic different from $2$ (because then $2^n$ is invertible, see \cite[beginning of p. 285]{Guo97}), the element $\alpha$ is unipotent, hence conjugate 
by an element $y\in \GL_n(\E)$ to a unipotent element $\alpha'\in \mathcal{M}_n(\F)$ in Jordan form (i.e. with nilpotent part in 
Jordan form). Setting $h=g(y,0)\in \H$, then $u'=h u h^{-1}= g(\alpha',\beta')\in \U_\S$, so we could suppose from the beginning 
that $\overline{\alpha}=\alpha$, which is what we do. Moreover the relation $\alpha^2=\I_n+\beta\overline{\beta}$ given by Equation (\ref{equation equations of S}) tells us that $\beta\overline{\beta}$ is nilpotent, or equivalently that $X=g(0,\beta)$ is nilpotent (see \cite[computation before Lemma 2.2]{Guo97}). Now \cite[Lemmas 2.2 and 2.3]{Guo97} and their proof apply to $X$ in any characteristic, and imply that there is $g\in \GL_n(\E)$ such that $g\beta \overline{g}^{-1}$ is in Jordan form, and in particular with coefficients in $\F$. This reads $g\beta \overline{g}^{-1}= \overline{g}\overline{\beta} g^{-1}$ hence 
$x\beta=\overline{\beta}\overline{x}$ for $x:=\overline{g}^{-1}g\in \GL_n(\E)$. This in turn implies that $x\alpha^2=\overline{\alpha^2}x$ because $\alpha^2=\I_n+\beta\overline{\beta}$. Because $\overline{\alpha}=\alpha$ we deduce that $x$ and $\alpha^2$ commute, but because the characteristic of 
$\F$ is not $2$ and $\alpha$ and $\alpha^2$ are unipotent, the element $\alpha$ is a polynomial in $\alpha^2$, hence 
$x$ commutes with $\alpha$. Now 
\[g(x,0)g(\alpha,\beta)=g(x\alpha,x\beta)=g(\alpha x,\overline{\beta}\overline{x})=g(\alpha,\overline{\beta})g(x,0)= 
g(\overline{\alpha},\overline{\beta}) g(x,0)
\] so that $u$ is indeed conjugate to $\overline{u}$ by $g(x,0)\in \H$.
\end{proof}

\begin{lemma}{\cite[Lemma 3.3, second statement]{Guo97}} 
If $g\in \H \U_{\S} w\H$, then $g^{-1}\in \H \U_{\S} w\H$.
\end{lemma}
\begin{proof}
It is enough to prove that if $u$ belongs to $\U_\S$, then $w^{-1}u^{-1}\in \H u w\H$, i.e. 
$u^{-1}\in w\H u \H w=\H wu w \H$. However from Lemma \ref{lemma Guo 3.2 (2)} we have 
$\H u^{-1}\H=\H u \H$, hence we need to show that $u\in \H w u w \H$, i.e. that 
$\rho(w u w)=(w uw )^2= w u^2 w=  \overline{u}^2$ is conjugate to $\rho(u)=u^2$ by an element of $\H$, which follows from 
Proposition \ref{proposition Guo prop 2 mais plus fort}. 
\end{proof}

In particular we have extended the following to non-archimedean local fields of characteristic $\neq 2$:

\begin{proposition}{\cite[Proposition 3]{Guo97}}\label{proposition Guo prop 3}
If $g\in \G$, then $g^{-1}\in \H g \H$.
\end{proposition}

\subsection{Descent and double cosets}\label{subsection descent and double cosets}

This paragraph is a preparation for our Galois descent argument. Let $\Gamma$ be a finite group acting on a group $\tK$ via group automorphisms. Let $\tK'$ be a
 $\Gamma$-stable subgroup of $\tK$. Set $\K:={\tK}^\Gamma$ and 
 $\K':=\tK'^\Gamma$.

 \begin{lemma}\label{lemma shaun} For $k$ in $\K$, if the non-abelian
   cohomology set $\H^1 (\Gamma , k\tK' k^{-1}\cap \tK' )$ is trivial, then we
   have
  \[
   (\tK' k \tK')^\Gamma = \K' k \K'  .
   \]
 \end{lemma}
 
\begin{proof} This is a standard cocycle calculation. An
 ``abstract'' proof may be found in \cite{St01}, Lemma 2.1. 
\end{proof}

From this we deduce the following second lemma:

  \begin{lemma}\label{lemma Galois descent} With the notation as above, assume that $\kappa$ is an anti-involution of $\tK$ which fixes all double 
  cosets $\tK' \tilde{k} \tK' \subset \tK$ and that:

\begin{enumerate}
\item for all $\sigma\in \Gamma$, $\sigma$ and $\kappa$ commute,
\item for all $k\in \K$, $\H^1 (\Gamma , k\tK' k^{-1}\cap \tK' )=\{1\}$.
\end{enumerate}

Then all double cosets $\K' k \K'\subset \K$ are fixed by $\kappa$.
  \end{lemma}
\begin{proof}  Since the actions of $\Gamma$ and $\kappa$ commute,
  $\kappa$ stabilizes $\K$ and $\K'$. For $k\in \K$, we have $\kappa (k) \in \tK' k\tK'$ by hypothesis. By Lemma 
  \ref{lemma shaun}, this implies: \[\kappa (k)\in (\tK' k\tK' )\cap \K =(\tK' k\tK' )^\Gamma=\K' k\K'.\] 
\end{proof}

\subsection{GK property for PTB pairs}
 Here we deduce the generalization of Guo's result from Proposition \ref{proposition Guo prop 3} and Lemma \ref{lemma Galois descent}.  We refer to \cite{Reiner03} for standard facts about the Hasse invariant of central simple algebras used in the proof. 
 
\begin{theorem}\label{theorem stability of double cosets}
Let $(\G,\H)$ be of PTB type. The involution $i:g\mapsto g^{-1}$ stabilizes each double coset $\H g \H\subset \G$.
\end{theorem}
\begin{proof}
Our proof is by induction on $d$. If $d=1$, the statement is Proposition \ref{proposition Guo prop 3}. We now suppose that 
the result is true for all pairs $(\G',\H')$ of PTB type with index $1\leq d'<d$. The index of $\A$ is the denominator of 
the Hasse invariant \[\inv_\F(\A)=\frac{r}{d}\in \mathbb{Q}/\mathbb{Z}\] (where $\mathrm{gcd}(r,d)=1$). Moreover if $\L$ is a finite extension of $\A$, then \[\inv_\L(\A\otimes_\F \L)=\frac{[\L:\F]r}{d}\in \mathbb{Q}/\mathbb{Z}.\] It follows that one can choose a Galois extension $\L$ linearly disjoint of $\E$ over $\F$, such that the index $d'$ of $\A \otimes_\F \L$ is smaller than $d$: if $d$ is not a power of $2$ take $\L$ any Galois extension of $\F$ of degree an odd prime factor of $d$ (for example unramified) whereas if $d$ is a power of $2$, take $\L$ to be a quadratic extension of $\F$ with ramification opposite to that of $\E/\F$. Now set $\tG=(\A\otimes_\F \L)^\times$ and 
$\tH=(\A_\E\otimes_\F \L)^\times=((\A\otimes_\F \L)_\E)^\times$. The pair $(\tG,\tH)$ still of PTB type with involution $\theta=\mathrm{Ad}(\delta)$ again. Hence by induction $i$ acts as the identity on 
$\tH \backslash \tG/\tH$. Moreover $i$ commutes with the elements of $\Gamma:=\Gal(\L/\F)$. Take $g\in \G$. Since $g\tH g^{-1}\cap \tH=(g\A_\E g^{-1}\cap \A_\E)\otimes_{\F} \L$ and because $g\A_\E g^{-1}\cap \A_\E$ is an $\F$-algebra, by Lemma \ref{lemma trivial cohomology} we have $\H^1 (\Gamma , g\tH g^{-1}\cap \tH )=\{1\}$. Applying Lemma \ref{lemma Galois descent} then concludes the proof of the theorem.
\end{proof}

Together with Proposition \ref{proposition double cosets stability implies GK type}, Theorem \ref{theorem stability of double cosets} implies:

\begin{corollary}\label{corollary PTB implies GK}
A pair of PTB type is of GK type.
\end{corollary}

\begin{remark}
When $\F$ has characteristic zero, the result above can also be deduced from the criteria given in \cite{AG09}. In fact this is the method that we shall use in Section \ref{subsection GK archimedean}.
\end{remark}

\section{Self-duality and multiplicity one}\label{section self-duality and multiplicity one}

In this section we show that if $(\G,\H)$ is of PTB type then it is of GK type, and if moreover $\pi\in \Irr_{\H\dist}(\G)$ then $\pi=\pi^\vee$. The reason why we do this is that if the index $d$ of $(\G,\H)$ is greater than $2$, then there is no obvious analog of Guo's anti-involution $\tau$ of $\G$ stabilizing $\H$ and satisfying $\pi^\vee(g)\simeq \pi(\tau(g^{-1}))$ for $\pi \in \Irr(\G)$. Indeed setting $\G=\A^\times$, the anti-involution $\tau$ of \cite[p.3]{Guo97} is the the restriction of an anti-involution of $\A$, but if $d>2$ then $\A$ possesses no anti-involution otherwise its class in the Brauer group of $\F$ would be of order $\leq 2$, which is not the case by hypothesis. So we do not proceed as in \cite{Guo97} to directly deduce multiplicity one from the GK type property. We start with the cuspidal case where 
a globalization argument of \cite{PSP08} allows us to deduce the self-duality result from the $\F$-split case. We then deduce the general case  from the cuspidal one, 
by proving self-duality for $\H$-distinguished standard modules using the geometric lemma Bernstein and Zelevinsky ($p$-adic Mackey theory).
First we recall the following results from \cite{JR96} and \cite{Guo97} for use in the cuspidal case:

\begin{theorem}\label{theorem self-duality for JR and Guo pairs}{$[$\cite{JR96},\cite{Guo97}$]$.}
Let $(\G,\H)$ be of JR type or of Guo type, then all $\pi\in \Irr_{\H\dist}(\G)$ are self-dual.
\end{theorem}

\subsection{The cuspidal case}
In this section the pair $(\G,\H)$ is of PTB type over $\F$, with $\G=\A^\times$ and $\H=\A_\E^\times$.

\begin{proposition}\label{proposition self duality of dist cuspidal reps}
Let $\rho$ be an $\H$-distinguished cuspidal representation of $\G$, then $\rho$ is self-dual.
\end{proposition}
\begin{proof}
Let $\mathfrak{e}/\mathfrak{f}$ be a quadratic extension of global fields and $v_0$ a place of $\mathfrak{f}$ which remains non split in $\mathfrak{e}$, such that 
$\mathfrak{e}_{v_0}\simeq \E$ and 
$\mathfrak{f}_{v_0}\simeq \F$. Let $\mathfrak{A}$ be a central simple $\mathfrak{f}$-algebra such that $\mathfrak{A}_{v_0}\simeq \A$, and if the characteristic
of $\F$ is positive we moreover require that $\mathfrak{A}$ is a central division $\mathfrak{f}$-algebra (this is possible thanks to the 
so-called Brauer-Hasse-Noether theorem for which we refer to \cite[Theorem 1.12]{PR91}). Let $\mathbb{A}_{\mathfrak{f}}$ be the ring of adeles of 
$\mathfrak{f}$, then by \cite[Theorem 4.1]{PSP08} and \cite[Theorem 1.3]{GL18} there is a cuspidal automorphic 
representation $R$ of $\mathfrak{G}=(\mathfrak{A}\otimes_{\mathfrak{f}} \mathbb{A}_{\mathfrak{f}})^\times$ such that 
 $\rho=R_{v_0}$ which has non-zero period with respect to $\mathfrak{H}=\mathfrak{G}_{\mathfrak{e}^\times}$. 
 In particular for any finite place $v$ of $\mathfrak{f}$, the local component 
 $R_{v}$ is $\mathfrak{H}_{v}$-distinguished. If moreover $v$ is such that $\mathfrak{G}_v$ is split over $\mathfrak{f}_v$ (which is the case of all finite places except possibly a finite number), 
 it follows from the two cases of Theorem \ref{theorem self-duality for JR and Guo pairs} depending on whether
 $v$ splits or not inside $\mathfrak{e}$ that $R_v$ is self-dual. One can then apply the strong multiplicity one theorems of 
 \cite{BaduGJL08} and \cite{BR17} to deduce that $R$ is self-dual, hence $\rho=R_{v_0}$ as well.
\end{proof}

\subsection{Reminder on the Langlands classification for $\Irr(\G)$}\label{subsection langlands classification}

Here we recall important facts from \cite{BZ77}, \cite{S78}, \cite{DKV84}, \cite{T90} and \cite{Badu02}. For $\pi$ a smooth representation of $\GL_n(\D)$ with 
$\D$ a central division $\F$-algebra, we set $n=n(\pi)$.  

Let $\rho$ be a cuspidal representation 
of $\GL_n(\F)$ and $\nu$ the character of $\GL_n(\F)$ defined in Section \ref{section notation}. Take $a\leq b\in \mathbb{R}$ with $b-a\in \mathbb{N}$, then the normalized parabolic induction 
\[\nu^{a}\rho\times \nu^{a+1}\rho\times \dots \times \nu^{b}\rho\] has a unique irreducible quotient which we denote by 
$\DD=[\nu^{b}\rho,\dots,\nu^{a}\rho]$. This quotient $\DD$ uniquely determines the sequence 
$(\nu^{a}\rho,\dots,\nu^{b}\rho)$ and we call the non-negative integer $l=b-a+1$ the length of $\DD$. We call such a representation $\Delta$ a segment of $\GL_{ln}(\F)$. 
Now if $\rho$ is a cuspidal representation of $\GL_n(\D)$ for $\D$ and $\F$-division algebra of index $d$, the image of $\rho$ by the Jacquet-Langlands transfer is a segment of $\GL_{nd}(\F)$, the length of which we denote by $s_\rho$. We then 
set $\nu_\rho=\nu^{s_\rho}$ for $\nu$ the positive character of $\GL_m(\D)$ defined in Section \ref{section notation}. Again 
the normalized parabolic induction 
\[\nu_\rho^{a}\rho\times \nu_\rho^{a+1}\rho\times \dots \times \nu_\rho^{b}\rho\] has a unique irreducible quotient which we denote by 
$\DD=[\nu_\rho^{b}\rho,\dots,\nu_\rho^{a}\rho]$, and call $l=b-a+1$ the length of $\DD$ again. The representation $\DD$ still uniquely determines the sequence $(\nu_\rho^{a}\rho,\dots,\nu_\rho^{b}\rho)$. If $\Delta$ and $\Delta'$ are two segments  
(say of $\GL_n(\D)$ and $\GL_{n'}(\D)$), we say that $\Delta$ precedes $\Delta'$ if one can write $\Delta=[\nu_\rho^{b}\rho,\dots,\nu_\rho^{a}\rho]$ and $\Delta'=[\nu_\rho^{b'}\rho,\dots,\nu_\rho^{a'}\rho]$ with $b-b'\in \mathbb{Z}$, $b\leq b'$, $b \geq a'-1$ and $a<a'$. 


The Langlands classification of irreducible representations of $\G=\GL_n(\D)$ asserts that if $(\DD_1,\dots,\DD_r)$ is a sequence of 
segments such that $\DD_i$ does not precede $\DD_j$ if $i<j$, then 
\[\DD_1\times \dots \times \DD_r\] has a unique irreducible quotient 
\[\pi=\L(\DD_1,\dots,\DD_r),\] that $\pi$ determines the multi-set 
$\{\DD_1,\dots,\DD_r\}$ and that every irreducible representation of $\G$ can be obtained in this manner. We call 
\[\DD_1\times \dots \times \DD_r\] the \textit{standard module} over $\pi$. 

Finally we recall that if we can write a segment 
\[\DD=[\DD_1,\dots,\DD_t]\] of $\G$ as a concatenation of sub-segments $\DD_i$, and 
$\P=\P_{(n(\DD_1),\dots,n(\DD_t))}$ is the standard parabolic subgroup of $\G$ containing all upper triangular matrices in $\G$ associated with the 
partition $(n(\DD_1),\dots,n(\DD_t))$ of $n$, with standard Levi decomposition \[\P=\M_{(n(\DD_1),\dots,n(\DD_t))} \N_{(n(\DD_1),\dots,n(\DD_t))}=\M \N,\] then the normalized Jacquet module $r_{\M,\G}(\DD)$ is given by the formula 
\[r_{\M,\G}(\DD)=\DD_1\otimes \dots \otimes \DD_t.\] Otherwise if $\M=\M_{(n_1,\dots,n_r)}$ is such that the partition 
$(n_1\dots,n_r)$ of $n$ is not associated to such a concatenation then 
\[r_{\M,\G}(\DD)=0.\]

\subsection{The double classes $\P\backslash \G/ \H$ and distinction of induced representations}

Here we recall results from \cite{O17} and \cite{C18}.

We fix $\G=\GL_n(\D)$ with $n=2m$ if the index $d$ of $\D$ is odd. When $d$ is even we denote by $e_\d$ the matrix $\d.\I_n$ whereas when $d$ is odd we set 
\[e_\d=\begin{pmatrix} & & & & & \d^2\\ 
& & & & \iddots & \\ & & & \d^2 &  & \\ & & 1 & &  & \\
& \iddots & & &  & \\ 1 & & & &  &   \end{pmatrix}.\] In both cases $\G=\H^\theta$ for $\theta=\Ad(e_\d)$. We denote by 
$\P$ the standard parabolic subgroup of $\G$ associated to the partition 
$\overline{n}=(n_1,\dots,n_r)$ of $\G$. When $d$ is even, we denote by $\I(\overline{n})$ the set of $r\times r$ symmetric matrices 
with entries in $\mathbb{N}$ such that the sum of the entries of the $i$-th row is equal to $n_i$, whereas 
if $d$ is odd we denote by $\I(\overline{n})$ the set of $r\times r$ symmetric matrices defined by the same condition plus the requirement that each diagonal coefficient is even. 
In \cite[Sections 2.2 ane 3.2]{C18}, Chommaux associates to each $s\in \I(\overline{n})$ an element that we denote $u_s$ in this paper, and the set 
$(u_s)_{s\in \I(\overline{n})}$ provides a set of representatives of the double-quotient $\P\backslash \G/ \H$. We do not need to describe $u_s$ explicitly here. However we notice that any 
\[s=(n_{i,j})\in \I(\nn)\] naturally defines the subpartition of $\nn$ given by 
$(n_{1,1},n_{1,2},\dots,n_{r,r-1},n_{r,r})$ where the zero elements have been omitted from this ordered sequence. We denote by 
$\P_s=\M_s\N_s$ the standard parabolic subgroup of $\G$ associated to this subpartition of $\nn$. We set $\theta_s$ the involution of $\G$ defined by 
\[\theta_s(x)=u_s\theta(u_s^{-1} x u_s)u_s^{-1},\] the fixed points of which are the group $u_s\H u_s^{-1}$. Then $\P_s$, $\M_s$ and $\N_s$ are $\theta_s$-stable and 
\[\P\cap u_s\H u_s^{-1}=\P_s^{\theta_s}=\M_s^{\theta_s}\N_s^{\theta_s}.\] 
When $d$ is even one has 
\[\M_s^{\theta_s}=\{\diag(a_{1,1},a_{1,2}\dots,a_{r,r-1},a_{r,r}),\ a_{j,i}=\theta(a_{i,j})\}\] whereas 
\[\M_s^{\theta_s}=\{\diag(a_{1,1},a_{1,2}\dots,a_{r,r-1},a_{r,r}),\ a_{i,i}=\theta(a_{i,i}), \ a_{j,i}=wa_{i,j}w^{-1} \ \mathrm{for} \ i\neq j\}\] for $w$ the anti-diagonal long Weyl element when $d$ is odd. We have the following relation between modulus characters: 
\begin{equation}\label{equation modulus even and odd} {\d_{{\P_s}^{\theta_s}}}_{\mid \M_s}= {\d_{\P_s}^{1/2}}_{\mid \M_s}\end{equation} 

\begin{remark}\label{remark correction Marion}
Note that in \cite[Proposition 2.3]{C18}, when $d$ is even, the author obtains the equality ${\d_{{\P_s}^{\theta_s}}}_{\mid \M_s}= {\d_{\P_s}}_{\mid \M_s}$ which is not correct. This is due to the following minor oversight: with 
notations of ibid. 
one should have \[\d_{{\P_s}^{\theta_s}}(t)=\prod_{\{\alpha\in \Phi^+-\Phi_s^+\}} |\alpha(t)|_{\E}^{d/2}\] and 
\[\d_{\P_s}(t)=\prod_{\{\alpha\in \Phi^+-\Phi_s^+\}} |\alpha(t)|_{\E}^{d}.\] This does not affect the rest of Section 2 in \cite{C18}, hence it does no affect the other results of \cite{C18} either.
\end{remark}

Finally \cite[Theorem 1.1]{O17} together with the computation of Jacquet modules of segments and Equation 
(\ref{equation modulus even and odd}) has the following consequence: 

\begin{proposition}\label{proposition sufficient condition for distinction} If the representation 
$\DD_1\times \dots \times \DD_r$ of $\G=\GL_n(\D)$ is $\H$-distinguished, then setting 
$n_i=n(\DD_i)$, there exist $s=(n_{i,j})\in \I(\nn)$, segments $\DD_{i,j}$ with  
$n(\DD_{i,j})=n_{i,j}$ and $\DD_i=[\DD_{i,1},\dots,\DD_{i,r}]$ such that 
\[\DD_{1,1}\otimes \DD_{1,2} \otimes \dots \otimes \DD_{r,1}\otimes \DD_{r,r}\] is 
$\M_s^{\theta_s}$-distinguished. This latter condition is equivalent to $\DD_{j,i}\simeq \DD_{i,j}^\vee$ if 
$i\neq j$ and $\DD_{i,i}$ is $\GL_{n_{i,i}}(\D)_{\E^\times}$-distinguished. 
\end{proposition}

\subsection{Self-duality in $\Irr_{\H\dist}(\G)$ and multiplicity one}

Here $\G$ and $\H$ are as in the previous section, i.e. $(\G,\H)$ is of PTB type. We first extend the self-duality result of $\H$-distinguished representations from the case of cuspidal representations to the case of 
segments (which are known to be the essentially square-integrable representations of $\G$).

\begin{proposition}\label{proposition self duality of dist segments}
Let $\DD$ be an $\H$-distinguished segment of $\G$, then $\DD$ is self-dual.
\end{proposition}
\begin{proof}
We recall that $\DD$ can be written as the irreducible quotient of a representation 
\[\nu_\rho^{(1-k)/2}\rho \times \dots \times \nu_\rho^{(k-1)/2}\rho\] with $\rho$ cuspidal and $k\geq 2$ as we already know the cuspidal case. If $\DD$ is distinguished then 
\[\nu_\rho^{(1-k)/2}\rho \times \dots \times \nu_\rho^{(k-1)/2}\rho\] is. Note that written this way, 
because the central character of an $\H$-distinguished representation must be trivial the representation $\rho$ 
must have a unitary central character. We can then apply Propostion \ref{proposition sufficient condition for distinction} 
to this induced representation, but as the segments $\DD_i$ are cuspidal representations and as the central character of
$\nu_\rho^{\alpha}\rho$ is $\nu_\rho^{\alpha}c_\rho$ with $c_\rho$ unitary, it follows that there is only one element $s$ 
which will contribute to distinction and that this element is the partition $s=(n_{i,j})$ with $n_{i,j}=0$ unless $j=n+1-i$, 
in which case $n_{i,n+1-i}=n(\rho)$. We conclude that $\rho$ is distinguished hence self-dual 
(by Proposition \ref{proposition self duality of dist cuspidal reps}) when $k$ is odd, whereas
it implies that $\rho$ is self-dual when $k$ is even. In both cases $\rho$ is self-dual, whence $\DD$ is self-dual.
\end{proof}

We can now move on to the general case.

\begin{theorem}\label{theorem self duality of dist irreps}
If $\pi\in \Irr_{\H\dist}(\G)$, then $\pi$ is self-dual. 
\end{theorem}
\begin{proof}
Take $\DD_1\times \dots \times \DD_r$ the standard module lying over $\pi$, it is distinguished because $\pi$ is. In 
view of Propositions \ref{proposition sufficient condition for distinction} and \ref{proposition self duality of dist segments}, it follows 
verbatim from the proof of \cite[Lemma 3.3 and Proposition 3.4]{G15} up to the notational difference that one replaces conjugate self-duality by self-duality, 
that the standard module $\DD_1\times \dots \times \DD_r$ is self-dual. This in turns implies that 
$\pi$ is self-dual.
\end{proof}

We then deduce from Theorem \ref{theorem self duality of dist irreps} and Corollary \ref{corollary PTB implies GK} the following
result.

\begin{corollary}\label{corollary PTB is G}
The pair $(\G,\H)$ is of Gelfand type and if $\pi\in \Irr_{\H\dist}(\G)$ then $\pi$ is self-dual.
\end{corollary}

\section{The archimedean case}\label{section archimedean case}

\subsection{Smooth admissible Fr\'echet representations}

Let us fix the category of representations that we are working in. Let $\G$ be a real reductive 
group as in \cite[2.1]{W88}, and $\K$ be a maximal 
compact subgroup of $\G$. We denote by $\G^0$ the neutral connected component of $\G$ and by $\K^0$ that of $\K$. We will say that 
$(\pi,V)$ is \textit{a representation of $\G$} if it is a smooth Fr\'echet $\G$-module of moderate growth, and if its subspace $V_0$ of $\K$-finite vectors is an 
admissible $(\mathfrak{g},\K)$-module (\cite[Chapter 1]{W88} and \cite[Chapter 11]{W92}). Such a representation 
can always be obtained as the space 
of smooth vectors of a continuous admissible representation of $\G$ on a Hilbert space (\cite[Chapter 11]{W92}), and this 
will allow us to appeal to results from \cite{K86} when considering the restricion to $\G^0$ of a representation of $\G$ (note that \cite{K86} only deals 
with connected real reductive groups). 
Morphisms between representations of $\G$ will be the continuous $\G$-intertwining operators the image of which are topologically closed summands, and we denote by 
$\SF(\G)$ the corresponding category of representations of $\G$. We will say that a representation $(\pi,V)$ of $\G$ is irreducible if its underlying 
admissible $(\mathfrak{g},\K)$-module $V_0$ is irreducible. We denote by 
$\Irr^{\SF}(\G)$ the set of isomorphism classes of irreducible representations of $\G$. If $\H$ is a closed subgroup of $\G$
 and $\pi\in \SF(\G)$, we denote by $\Hom_{\H}(\pi,\CC)$ the space of continuous $\H$-invariant linear forms 
 on the space of $\pi$, and by $\Irr_{\H\dist}^{\SF}(\G)$ the classes of representations $\pi$ 
in $\Irr^{\SF}(\G)$ which satisfy $\Hom_{\H}(\pi,\CC)\neq \{0\}$. We recall that for $(\pi,V)\in \SF(\G)$, the representation $(\pi^\vee,V^\vee)$ of $\G$ on the space of 
smooth vectors in the space $V'$ of continuous linear forms on $V$ also belongs to $\SF(\G)$, 
and that $(\pi^\vee)^\vee\simeq \pi$. We denote by $\Sc(\G)$ the space of Schwartz (complex valued, smooth, with all derivatives rapidly decreasing) functions on $\G$, and 
recall that it naturally inherits the structure of a Fr\'echet space. We call a Schwartz distribution on
$\G$ a continuous linear form from $\Sc(\G)$ to $\CC$, and denote by $\SD(\G)$ the space of Schwartz distributions on $\G$. The space $\SD(\G)$ is naturally equipped with right and left actions of $\G$. We 
recall from \cite[11. 8]{W92} that if $(\pi,V)\in \SF(\G)$, the Schwartz algebra $\Sc(\G)$ acts on 
$\pi$ by the formula \[\pi(\phi)v=\int_{\G}\phi(g)\pi(g)v dg\] for $\phi\in \Sc(\G)$ and $v\in V$, where the integral converges in $V$. A representation (which we confuse as often with its isomorphism class) of $\G$ belongs to $\Irr^{\SF}(\G)$ if and only if it is an irreducible (in the algebraic sense) $\Sc(\G)$-module. There is also another 
convolution sub-algebra of $\Sc(\G)$ which will turn out to be useful for us, namely the algebra $\mathcal{C}_{c,\K\mbox{-}\mathrm{fin}}^\infty(\G)$ of 
smooth compactly supported functions on $\G$, which are left and right $\K$-finite. If $(\pi,V)$ is a representation of $\G$, note that the restriction $(\pi^0,V)$ of 
$\pi$ to $\G^0$ is also a representation of $\G^0$, and that the space of 
$\K$-finite vectors $V_0$ of $\pi$ is equal to that of $\K^0$-finite vectors of $\pi^0$. We will make use of the following result.

\begin{proposition}\label{proposition Knapp}
Let $(\pi,V)$ be an irreducible representation of $\G$, then $\pi(\mathcal{C}_{c,\K\mbox{-}\mathrm{fin}}^\infty(\G)).V=V_0$ and $V_0$ is an irreducible 
$\mathcal{C}_{c,\K\mbox{-}\mathrm{fin}}^\infty(\G)$-module.
\end{proposition}
\begin{proof}
It is immediate that $\pi(\mathcal{C}_{c,\K\mbox{-}\mathrm{fin}}^\infty(\G)).V\subset V_0$. 
The equality will follow from the second part of the statement. Because $V_0$ is an irreducible $(\mathfrak{g}, \K)$-module, it is sufficient to show that 
if $v_0$ is a non-zero vector in $V_0$, the sub-$(\mathfrak{g},\K)$-module $\pi(\mathcal{C}_{c,\K\mbox{-}\mathrm{fin}}^\infty(\G)).v_0$ is not reduced to zero. However, setting $\U(\mathfrak{g})$ for 
the envelopping algebra of the complexified Lie algebra of $\G$, it follows from 
\cite[Proposition 9.5]{K86} that $\pi(\mathcal{C}_{c,\K^0\mbox{-}\mathrm{fin}}^\infty(\G^0)).v_0= \pi(\U(\mathfrak{g}))v_0$, 
and $v_0\in \pi(\U(\mathfrak{g}))v_0$ because $\U(\mathfrak{g})$ has a unit so $\pi(\mathcal{C}_{c,\K^0\mbox{-}\mathrm{fin}}^\infty(\G^0)).v_0$ is nonzero. To conclude, we simply observe that the extension of functions by zero outside $\G^0$ embeds $\mathcal{C}_{c,\K^0\mbox{-}\mathrm{fin}}^\infty(\G^0)$ as a sub-algebra 
of $\mathcal{C}_{c,\K\mbox{-}\mathrm{fin}}^\infty(\G)$.
\end{proof}

\begin{corollary}\label{corollary Knapp}
Let $(\pi,V)$ and $(\pi',V')$ be two irreducible representations of $\G$, and $v\in V$ and $v'\in V'$ two non-zero vectors. There exists $\phi\in \mathcal{C}_{c,\K\mbox{-}\mathrm{fin}}^\infty(\G)$ such that $\pi(\phi)v\neq 0$ and $\pi'(\phi)v'\neq 0$.
\end{corollary}
\begin{proof}
By Proposition \ref{proposition Knapp}, there is $\phi\in \mathcal{C}_{c,\K\mbox{-}\mathrm{fin}}^\infty(\G)$ such that $\pi(\phi)v\neq 0$. If $\pi'(\phi)v'\neq 0$ we are done. If not by Proposition \ref{proposition Knapp} there is $\phi'\in \mathcal{C}_{c,\K\mbox{-}\mathrm{fin}}^\infty(\G)$ such that $\pi'(\phi')v'\neq 0$. If $\pi(\phi')v\neq 0$ we are done as well. If not the function 
$\phi+\phi'$ satisfies the required property.
\end{proof}

\subsection{The GK property for archimedean pairs of PTB type}\label{subsection GK archimedean}

The only archimedean pairs of PTB type are the pairs $(\G,\H)$ with $\G=\GL_{2n}(\RR)$ or $\G=\GL_n(\Ha)$ and $\H=\GL_n(\CC)$ which 
are the pairs that Guo 
considers in the non-archimedean case. 
We recall that we consider $\H$ as the centralizer $\G_{\CC^\times}$ of $\CC^\times$ in $\G$, where we embedded $\CC$ as an $\RR$-sub-algebra of 
$\mathcal{M}_{2n}(\RR)$ or $\mathcal{M}_{n}(\Ha)$ depending on the pair, all such embeddings being $\G$-conjugate. We again denote by $\theta$ the natural inner involution attached to the pair $(\G,\H)$. Though Guo only considers non-archimedean local fields, his proof of stability 
of double cosets $\H g\H\subset \G$ under the involution $i:g\mapsto g^{-1}$ of $\G$ goes through in the archimedean case. However 
once we know this stability property, we could not find a reference to conclude that $(\G,\H)$ is of GK type when $\F=\RR$. So 
instead we appeal to \cite{AG09}, and check that the criteria given there apply to archimedean PTB pairs. In fact what we say hereafter is also 
valid for non-archimedean local fields of characteristic zero but we already proved the Gelfand pair property in a simple manner for non-archimedean pairs of PTB type.

\begin{definition}
Let $(\bG,\bH,\theta)$ be a symmetric pair defined over $\RR$ and $\sigma$ be the anti-involution $\sigma=\theta\circ i:g\mapsto \theta(g^{-1})$ of $\G$.
Following \cite[Definition 7.1.8]{AG09}, we say that the pair $(\G,\H)$ is of 
GK type if any distribution in $\SD(G)$ which is $\H$-bi-invariant is fixed by $\sigma$.
\end{definition}

The above definition of pairs of GK type is more restrictive than the one given in Definition \ref{definition GK pair} in the non-archimedean case. 

\begin{definition} With $(\G,\H)$ as above, we will say that $(\G,\H)$ is a Gelfand pair if for any $\pi\in \Irr^{\SF}(\G)$, the space 
$\Hom_{\H}(\pi,\CC)$ is of dimension at most one.
\end{definition}

Now we come back to $(\G,\H)$ of PTB type, and check that the pair $(\G,\H)$ is of GK type. According to 
\cite[Section E, Diagram]{AG09}, to prove that $(\G,\H)$ is of GK type, it is sufficient to check that if $(\bG, \bH, \theta)$ is a PTB pair defined over $\RR$, then all its descendants 
$(\bG',\bH',\theta')$ are regular (see \cite{AG09}, Section 7.4) and satisfy $\H^1(\RR,\bH')=\{1\}$. This latter condition is always satisfied thanks to 
Lemma \ref{lemma trivial cohomology}. The descendants of a PTB pair can be computed by a straightforward adaptation of 
\cite[Theorem 7.7.1]{AG09}, and turn out to be products of either diagonal pairs, Galois pairs or PTB pairs again. The first two types of
pairs are regular thanks to \cite[Theorem 7.6.5]{AG09} and \cite[Section E, Diagram]{AG09}. To show that 
PTB pairs are also regular, according to \cite[Section E, Diagram]{AG09}, it suffices to show that they are special (\cite[Definition 7.3.4]{AG09}). 

\begin{lemma}\label{lemma PTB special}
 Let $(\bG,\bH,\theta)$ be a PTB pair defined over $\RR$, then it is special.
\end{lemma}
\begin{proof}
It is enough to check that $(\bG,\bH,\theta)$ satisfies the hypothesis of 
\cite[Proposition 7.3.7]{AG09}. Note that there is a typo in the published version of \cite{AG09} (we thank D. Gourevitch for informing us) 
and we refer to the statement of the arxiv version \cite{AG15}. We denote by $\A$ the central simple $\RR$-algebra such that $\G=\A^\times$. 
Let $n=2m$ be the dimension of $\bG$ so that $\A_\CC=\A_i$ (where $i$ is a square root of $-1$ in $\CC$) is isomorphic to $\mathcal{M}_n(\CC)$, and consider the JR pair 
$(\GL_n,\GL_m\times \GL_m,\theta_{m,m})$ which is defined over $\RR$ too. Write $\theta=\Ad(i)$, where we see $i$ as a matrix inside $\G$. We have 
$\bG=\bG(\CC)=\GL(n,\CC)$, and inside $\GL_n(\CC)$, the element $i$ viewed as a matrix in $\G$ is conjugate by a matrix 
$\P$ to $i.\diag(\I_m,-\I_m)$. So the conjugation $\Ad(\P)$ 
provides an isomorphism over $\CC$ from $(\bG,\bH,\theta)$ to $(\GL_n,\GL_m\times \GL_m,\theta_{m,m})$. Following \cite{AG09}, we set 
\[\mathcal{M}_n(\CC)^{\sigma}=\{x\in \mathcal{M}_n(\CC), \ \theta(x)=-x\},\] and  
\[\mathcal{M}_n(\CC)^{\sigma_{m,m}}=\{x\in \mathcal{M}_n(\CC), \ \theta_{m,m}(x)=-x\}.\]
We define $\A^\sigma$ and $\mathcal{M}_n(\RR)^{\sigma_{m,m}}$ similarly. With notations as in \cite[Proposition 7.3.7]{AG09}, the space  
$\Q(\mathcal{M}_n(\CC)^{\sigma})$ turns out to be $\A^\sigma$, and the space $\Q(\mathcal{M}_n(\CC)^{\sigma_{m,m}})$ to be $\mathcal{M}_n(\RR)^{\sigma_{m,m}}$: the space of matrices in $\mathcal{M}_n(\CC)^{\sigma}$ commuting with $\bH$ is reduced to zero, and that of matrices in $\mathcal{M}_n(\CC)^{\sigma_{m,m}}$ commuting with $\GL_n(\CC)\times \GL_n(\CC)$ as well. Note that 
both $\A^\sigma$ and $\mathcal{M}_n(\RR)^{\sigma_{m,m}}$ have 
same dimension $2m^2$ over $\RR$ as their complexification are isomorphic via $\Ad(\P)$. Note that the nilpotent cone of $\mathcal{M}_n(\CC)^{\sigma}$ and 
$\mathcal{M}_n(\CC)^{\sigma_{m,m}}$ are also in bijection via $\Ad(\P)$. Finally for $x$ in the nilpotent cone of $\mathcal{M}_n(\CC)^{\sigma}$, defining $d_{\sigma}(x)$ 
as in \cite[Notation 7.1.12]{AG09}, we can define $d_{\sigma_{m,m}}(\Ad(\P)(x))$ to be equal to $\Ad(\P)(d_{\sigma}(x))$, and the centralizer of $\Ad(\P)(x)$ in $\mathcal{M}_n(\CC)^{\sigma_{m,m}}$ is conjugate by $\P$ to the centralizer of $x$ in $\mathcal{M}_n(\CC)^{\sigma}$. So it is equivalent to check that 
the hypothesis of \cite[Proposition 7.3.7]{AG09} are satisfied either for $(\bG,\bH,\theta)$, or for $(\GL_n,\GL_m\times \GL_m,\theta_{m,m})$, but for this second pair 
they are indeed satisfied thanks to \cite[Lemma 7.7.5]{AG09}. The statement of the lemma now follows because an element in the nilpotent cone of $\A^\sigma$ is in the nilpotent cone of $\mathcal{M}_n(\CC)^{\sigma}$.
\end{proof}

The outcome of this discussion is the following.

\begin{proposition}\label{proposition PTB over R is GK}
Let $(\G,\H)$ be an archimedean pair of PTB type, then it is of GK type.
\end{proposition}

\subsection{Self-duality and multiplicity one}

As in the non-archimedean case, we will say that the PTB pair $(\G,\H)$ is of Gelfand type if the space $\Hom_{\H}(\pi,\C)$ 
is of dimension at most one whenever $\pi\in \Irr^{\SF}(\G)$. 
In order to conclude that $(\G,\H)$ is in fact of Gelfand type, according to \cite[Section E, Diagram]{AG09}, it suffices to check that 
$\bG$ posseses an $\Ad(\bG)$-admissible anti-automorphism which stabilizes $\bH$. This anti-automorphism exists in 
both cases ($\G=\GL_{2n}(\RR)$ and $\G=\GL_{n}(\Ha)$) and it is the involution $\tau$ defined in \cite[p. 275]{Guo97}. We can in fact say more: by 
\cite[Theorem 8.2.1]{AG09}, if $\pi\in \Irr^{\SF}(\G)$, 
then $\pi^\vee=\pi\circ \tau\in \Irr^{\SF}(\G)$. Now following the last paragraph of 
\cite[p.67]{JR96} adapted to the archimedean setting, we obtain:

\begin{theorem}\label{theorem PTB over R is G}
Let $(\G,\H)$ be an archimedean pair of PTB type, then it is of Gelfand type and for all $\pi\in \Irr_{\H\dist}^{\SF}(\G)$, one has $\pi=\pi^\vee$.
\end{theorem}
\begin{proof}
We already noticed that $(\G,\H)$ is of Gelfand type because of the existence of an $\Ad(\bG)$-admissible anti-automorphism which stabilizes $\bH$. We now give more details on the adaptation of the argument of \cite[p.67]{JR96} to the archimedean setting. 
We denote by $\K$ a maximal compact subgroup of $\G$. We take 
$(\pi,V)$ a representative of an element in $\Irr_{\H\dist}^{\SF}(\G)$. Set $(\widetilde{\pi}, V)=(\pi\circ \tau, V)$ so that $\widetilde{\pi}$ is thus isomorphic to $\pi^\vee$ (in this proof we need to consider representations rather than isomorphism classes of representations), and take $\ell'\in \Hom_{\H}(\widetilde{\pi},\CC)$ and 
$\ell\in \Hom_{\H}(\pi,\CC)$ both non-zero. Then according to \cite[Lemma 2.1.6]{AGS08}, for $\phi\in \Sc(\G)$, the linear form 
\[\pi^*(\phi)\ell: v\mapsto \int_{\G}\phi(g)\ell(\pi(g^{-1})v)dg\] belongs to $W=V^\vee$ and the map 
$\phi\mapsto \pi^*(\phi)\ell$ is continuous map from $\Sc(\G)$ to $W$. We choose a non-zero isomorphism $\I$ from 
$\widetilde{\pi}$ to $\pi^\vee$, so the map \[\D_{\ell,\ell'}:\phi\mapsto \ell'(I^{-1}(\pi^*(\phi)\ell))\] is an $\H$-bi-invariant distribution 
in $\SD(\G)$, hence it is invariant under $\sigma$. For $\phi\in \Sc(\G)$, we write $\phi^\vee=\phi\circ i$, it is also in 
$\Sc(\G)$. Take $\phi$ of the form $\phi_1^\vee*\phi_2$ for $\phi_i\in \Sc(\G)$, then 
\[\D_{\ell,\ell'}(\phi)=\ell'(I^{-1}(\pi^*(\phi)\ell))=\ell'(I^{-1}(\pi^\vee(\phi_1^\vee)\pi^*(\phi_2)\ell))=\ell'(\widetilde{\pi}(\phi_1^\vee)I^{-1}(\pi^*(\phi_2)\ell)) \]
\[= \widetilde{\pi}^*(\phi_1)\ell'(I^{-1}(\pi^*(\phi_2)\ell)),\] but on the other hand 
\[\D_{\ell,\ell'}(\phi)=\D_{\ell,\ell'}(\phi^\sigma)=\D_{\ell,\ell'}(\phi_2^\sigma*\phi_1^\theta)=\widetilde{\pi}^*(\phi_2^\theta)\ell'(I^{-1}(\pi^*(\phi_1^\theta)\ell)).\] We now take $\phi_2$ of the form $\phi_3*\phi_4$ with $\phi_3$ and $\phi_4$ in 
$\Sc(\G)$, and we choose $\phi_4$ such that both $\lambda=\pi^*(\phi_4)\ell\neq 0$ in $V^\vee$ and 
 $\lambda'=\widetilde{\pi}^*(\phi_4^\theta)\ell'\neq 0$ in $V^\vee$ (this is always possible taking $\phi_4$ approximating the identity). The relation above implies that 
 \[\widetilde{\pi}^*(\phi_1)\ell'(I^{-1}(\pi^\vee(\phi_3)\lambda))= \widetilde{\pi}^\vee(\phi_3^\theta)\lambda'(I^{-1}(\pi^*(\phi_1^\theta)\ell)),\] so that if $\pi^\vee(\phi_3)\lambda=0$, then $\widetilde{\pi}^\vee(\phi_3^\theta)\lambda'=0$ as well. In particular there exists a non-zero linear map $\J$ from $\pi^\vee(\Sc(\G))\lambda$ to $\widetilde{\pi}^\vee(\Sc(\G))\lambda'=\pi^\vee(\Sc(\G))\lambda'$ such that \[\J(\pi^\vee(\phi)\lambda)=\widetilde{\pi}^\vee(\phi^\theta)\lambda'\] for all $\phi\in \Sc(\G)$. We now set $W=V^\vee$ and denote by $W_0$ its underlying irreducible $(\mathfrak{g},\K)$-module. Thanks to Corollary  \ref{corollary Knapp} applied to $\pi^\vee$ and $\widetilde{\pi}^\vee\circ \theta$, one can choose $\phi\in {\mathcal{C}_{c,\K\mbox{-}\mathrm{fin}}^\infty}(\G)$ such that $\pi^\vee(\phi))\lambda$ and $\widetilde{\pi}^\vee(\phi^\theta)\lambda'$ are non-zero so we could suppose from the beginning that 
 $\lambda$ and $\lambda'$ were both in $W_0$. Moreover as 
 \[\pi^\vee({\mathcal{C}_{c,\K\mbox{-}\mathrm{fin}}^\infty}(\G))\lambda=\pi^\vee({\mathcal{C}_{c,\K\mbox{-}\mathrm{fin}}^\infty}(\G))\lambda'=W_0\] thanks to Proposition \ref{proposition Knapp}, the map $\J$ sends $W_0$ to 
 $W_0$, and its restriction to $W_0$ can be checked to be a non-zero intertwining operator of $(\mathfrak{g},\K)$-modules from 
 $(\pi^\vee,W_0)$ to $(\widetilde{\pi}^\vee\circ \theta,W_0)$. For example if $X\in \mathfrak{g}$, one has 
\[\J(\pi(X)\pi(\phi)\lambda)= \J(\pi(L_X\phi)\lambda)=\widetilde{\pi}^\vee((L_X\phi')^\theta)\lambda'=
\widetilde{\pi}^\vee(X^\theta)\widetilde{\pi}^\vee(\phi^\theta)\lambda'=\widetilde{\pi}^\vee(X^\theta)\J(\pi(\phi)\lambda).\] To conclude we appeal to the automatic continuity 
 theorem  (\cite[Theorem 11.6.7, second statement]{W92}) and deduce that $\pi^\vee$ is isomorphic to 
 $\widetilde{\pi}^\vee\circ \theta\simeq \widetilde{\pi}^\vee\simeq \pi$. 
\end{proof}

\begin{remark}
The proof above also applies to archimedean pairs of JR type (i.e. pairs of the form $(\GL_{2n}(\F),\GL_n(\F)\times \GL_n(\F))$ with 
$\F=\RR$ or $\CC$) and shows that irreducible distinguished representations of $\GL_{2n}(\F)$ are self-dual. These pairs are proved to be of Gelfand type in \cite{AG09} but it does not seem that self-duality of 
irreducible distinguished representations is addressed in loc.cit. 
\end{remark}

\bibliographystyle{plain}
\bibliography{Mult1}
\end{document}